\documentclass[11pt]{article}
\usepackage{amssymb, amsmath, amsfonts, amsthm, graphics, mathrsfs, enumerate}
\usepackage{hyperref}
\usepackage{bm}
\hypersetup{
    colorlinks=true,
    linkcolor=blue,
    filecolor=magenta,      
    urlcolor=cyan,
}
\usepackage[hmargin=1 in, vmargin = 1 in]{geometry}
\usepackage{graphicx}
\usepackage{caption}

\newtheorem{lemma}{Lemma}

\newcommand{\rr}{\mathbf{r}}
\newcommand{\nn}{\mathbf{n}}
\newcommand{\dd}{\mathbf{d}}

\newcommand{\yy}{\mathbf{y}}
\newcommand{\cS}{\mathcal{S}}
\newcommand{\cD}{\mathcal{D}}
\newcommand{\cT}{\mathcal{T}}
\newcommand{\cF}{\mathcal{F}}
\renewcommand{\d}{\mathrm{d}}
\newcommand{\pd}{\partial}
\newcommand{\bs}{\boldsymbol}

\title{Robust fast direct integral equation solver for three-dimensional quasi-periodic scattering problems with a large number of layers}
\author{Bowei Wu\thanks{Department of Mathematical Sciences, UMass Lowell, Lowell, MA 01854. email: bowei\_wu@uml.edu}
and
Min Hyung Cho\thanks{Department of Mathematical Sciences, UMass Lowell, Lowell, MA 01854. email: minhyung\_cho@uml.edu}}

\begin{document}

\maketitle

\begin{abstract}
    A boundary integral equation method for the 3-D Helmholtz equation in multilayered media with many quasi-periodic layers is presented. Compared with conventional quasi-periodic Green's function method, the new method is robust at all scattering parameters. A periodizing scheme is used to decompose the solution into near- and far-field contributions. The near-field contribution uses the free-space Green's function in an integral equation on the interface in the unit cell and its immediate eight neighbors; the far-field contribution uses proxy point sources that enclose the unit cell. A specialized high-order quadrature is developed to discretize the underlying surface integral operators to keep the number of unknowns per layer small. We achieve overall linear computational complexity in the number of layers by reducing the linear system into block tridiagonal form and then solving the system directly via block LU decomposition.
The new solver is capable of handling a 100-interface structure with 961.3k unknowns to $10^{-5}$ accuracy in less than 2 hours on a desktop workstation.
\end{abstract}

\section{Introduction}\label{sec:intro}
Optical or electromagnetic waves in doubly-periodic multilayered media is one of the fundamental mechanisms in many modern high-tech devices such as dielectric gratings for high-powered laser\cite{perry1995high,barty2004overview}, thin-film photovoltaic \cite{atwater2011plasmonics,kelzenberg2010enhanced}, passive cooling devices using multilayer photonic structure \cite{raman2014passive}, photonic crystals \cite{jobook}, semiconductor packaging that is one of the hot topics in chip design \cite{deboi2022improved,kim2022emi}, and process control in semiconductor lithography \cite{model2008scatterometry} . Numerical simulations are often used to help design or optimize these devices where one must solve the scattering problem for various incident angles and/or wavelengths and repeat the computation for design optimization in many cases \cite{tsantili2018computational}. Therefore, it is imperative to have a robust and efficient solver. There are many well-known numerical methods, including finite-difference time-domain method \cite{taflove2005computational}, finite element method \cite{bao1995finite,monk2003finite,he2016spectral}, rigorous-coupled wave analysis or Fourier modal method \cite{moharam1981rigorous,li1996use,cho2008rigorous}, and integral equation method \cite{bruno2016windowed,bruno2017windowed,PerezArancibia2019DomainDF,nicholls2020sweeping,cho2012parallel,chen2018accurate,lai2014fast,cho2015robust,zhang2021fast,zhang2022fast,cho2019spectrally,boag1992analysis}. Each method has its own advantages and disadvantages. The integral equation method stands out with several very attractive benefits over other methods: the dimensionality of the problem is reduced with all the unknowns residing on the interfaces instead of in the volume, which significantly reduces the number of unknowns; the radiation condition is built into the Green's function and no artificial boundary conditions or perfectly matched layers are required; moreover, a problem can often be formulated as a Fredholm second-kind integral equation which is well-conditioned and suitable for an accelerated iterative matrix solver. However, due to the nature of the Green's function, discretization of the integral equation usually yields a dense matrix that is expensive to invert directly. Thus, large system solvers must be accelerated using fast linear algebra algorithms, such as the Fast Multipole Method \cite{greengard1987fast,rokhlin1990rapid,cho2010wideband} and the Fast Direct Solver \cite{hackbusch1999sparse,martinsson2005fast,greengard2009fast,borm2010efficient,xia2010fast}.

For two dimensional (2-D) quasi-periodic multilayered media, the boundary integral equation method combined with a periodizing scheme is used to build an efficient solver that can handle 1000s of layers \cite{cho2015robust}. The solver is further accelerated by the fast direct solver recently developed by Zhang and Gillman \cite{zhang2021fast,zhang2022fast} which can handle complex interfaces that require a large number of samples and is useful for parameter optimization. For three-dimensional (3-D) problems, to avoid challenges of surface integral quadrature, the Method of fundamental solution (MFS) has been used in the place of boundary integral equation for quasi-periodic multilayered media \cite{cho2019spectrally,boag1992analysis} and for doubly-periodic arrays of axisymmetric objects \cite{liu2016efficient}. However, the MFS approach has many limitations associated with the choice of artificial source points near the boundary and the ill- conditioning of the linear system, making it impractical as a general solver even though it was a good tool to show the effectiveness of the periodizing scheme in three dimensions. The difficulty of singular quadrature on layer interfaces has finally been overcome thanks to the recent development of the corrected trapezoidal quadrature \cite{wu2021zeta,wu2021corrected,wu2022unified}. This paper presents a robust and fast integral equation solver for the Helmholtz equation in 3-D quasi-periodic multilayered media using the direct solver based on Schur complement and block tridiagonal LU decomposition that was used for the 2-D problem \cite{cho2015robust}. One review paper \cite{kleemann2015fast} named the authors as one of the groups that can efficiently solve 3-D problems using the boundary integral equation method. Similar to the 2-D solver, the new 3-D solver's CPU time grows linearly with number of layers and is robust at Wood anomalies \cite{wood1902xlii}, making it possible to handle a large number of layers. This solver can be made highly efficient for optimization problems once it is accelerated by a fast direct solver that is a 3-D extension of \cite{zhang2022fast}.

\begin{figure}[t]
\includegraphics[width=\textwidth]{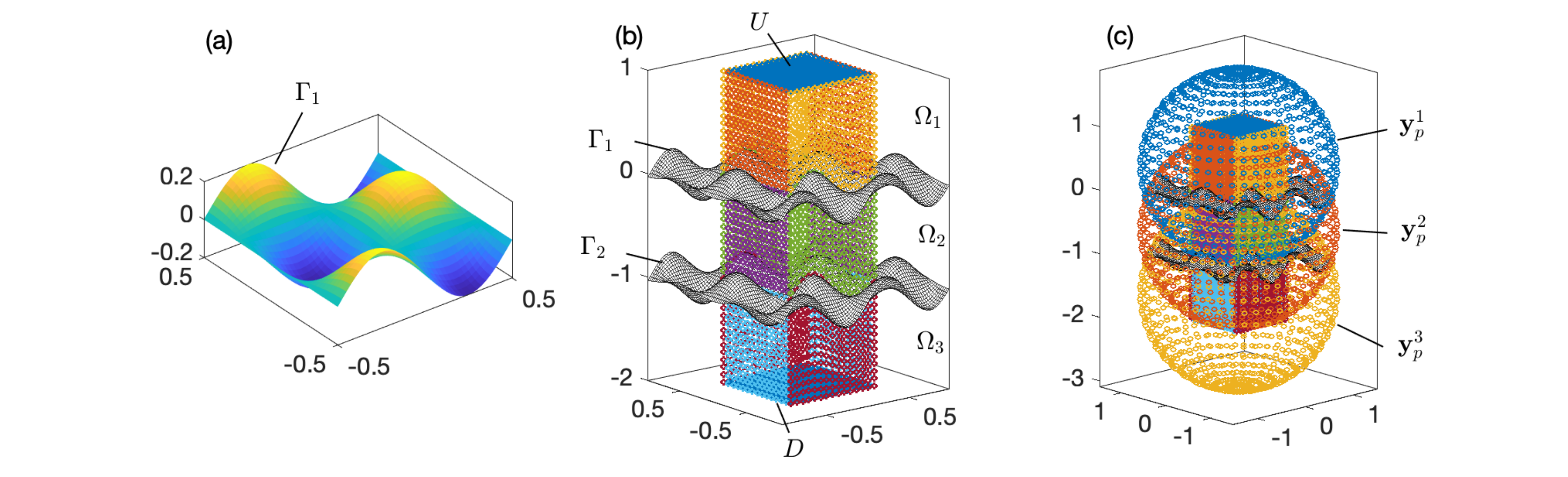}
\caption{(a) An interface parameterized by $g(x,y)=0.2\sin(2\pi x)\cos(2\pi y)$, with period $d=1$ in both directions. (b) Structure with multiple interfaces $\Gamma_i$ with the surrounding walls for each layer $\Omega_i$. (c) Proxy points $\{\mathbf{y}^i_p\}$ on spheres enclosing each layer $\Omega_i$.}\label{fig:periodization}
\end{figure}

The geometry of the problems (See Fig.~\ref{fig:periodization} for notation and schematics) consists of $I+1$ layers denoted by $\{\Omega_i\}_{i=1}^{I+1}$. There are $I$ surfaces $\{\Gamma_i\}_{i=1}^I$ where $\Gamma_i$ is the interface that separates $\Omega_i$ and $\Omega_{i+1}$. All interfaces have the same periodicity $d$ along $x$- and $y$-directions (see Fig.~\ref{fig:periodization}(b)). The wavenumber is $\{k_i\}_{i=1}^{I+1}$ in each layer. $L_i, R_i, F_i$, and  $B_i$ are the artificial side walls surrounding the unit cell of the layer $\Omega_i$ to impose quasi-periodic boundary conditions. $U$ and $D$ are, respectively, the artificial layers placed above $\Gamma_1$ at $z= z_u$ and below $\Gamma_I$ at $z=z_d$ to impose the radiation conditions. A plane wave is incident in the uppermost layer,
\begin{align}
u^{inc} (\mathbf{r})=\left\{\begin{array}{ll}e^{i\mathbf{k}\cdot \mathbf{r}}, & \mathbf{r} \in \Omega_1, \\0, & \mbox{otherwise,}\end{array}\right.
\end{align}
where the wave vector $\mathbf{k} =(k_{1x}, k_{1y}, k_{1z})= (k_1\sin{\phi^{inc}}\cos{\theta^{inc}}, k_1\sin{\phi^{inc}}\sin{\theta^{inc}}, k_1\cos{\phi^{inc}})$ with $0 \leq \theta^{inc} <  2\pi$ and $\pi/2 <\phi^{inc} < \pi$, and where $\mathbf{r} = (x,y,z)$. The incident wave is quasi-periodic (periodic up to a phase) in both directions, that is
\begin{align}
\alpha_x^{-1} u^{inc}(x+d, y, z) = \alpha_y^{-1}u^{inc}(x, y+d, z) = u^{inc}(x,y,z),
\end{align}
where the Bloch phases $\alpha_x$ and $\alpha_y$ are defined by
\begin{align}
\alpha_x = e^{idk_{1x}} \mbox{ and }\alpha_y = e^{idk_{1y}}.
\end{align}
From the standard scattering theory \cite{colton1998inverse}, the scattered wave $u_i$ must be quasi-periodic and satisfy the Helmholtz equation. Thus, the boundary value problem (BVP) for $u_i$ consists of the equation
\begin{align}
\Delta u_i(\mathbf{r}) + k_i^2 u_i(\mathbf{r}) = 0,~ \mathbf{r} \in \Omega_i \label{eq:helmholtz}
\end{align}
the continuity conditions at each interface $\Gamma_i$, $i=1,2,\cdots I$
\begin{align}
\begin{array}{llll}u_1-u_2 = -u^{inc} & \mbox{ and } & \frac{\partial u_1}{\partial \mathbf{n}}-\frac{\partial u_2}{\partial \mathbf{n}} = -\frac{\partial u^{inc}}{\partial \mathbf{n}} & \mbox{ on } \Gamma_1 \\u_i-u_{i+1} = 0 & \mbox{ and } &  \frac{\partial u_i}{\partial \mathbf{n}}-\frac{\partial u_{i+1}}{\partial \mathbf{n}} = 0 &  \mbox{ on } \Gamma_i, i=2,3,\cdots, I\end{array},\label{eq:continuity_cond}
\end{align}
the quasi-periodicity conditions on the side walls for all layers
\begin{align}
\begin{array}{lll}
u_i|_{L_i}=\alpha_x^{-1}u_i|_{R_i} & \mbox{and} & u_i|_{B_i}=\alpha_y^{-1}u_i|_{F_i} \\
\left. \frac{\partial u_i}{\partial \mathbf{n}}\right|_{L_i}=\alpha_x^{-1}\left.\frac{\partial u_i}{\partial \mathbf{n}}\right|_{R_i} & \mbox{and} & \left.\frac{\partial u_i}{\partial \mathbf{n}}\right|_{B_i}=\alpha_y^{-1}\left.\frac{\partial u_i}{\partial \mathbf{n}}\right|_{F_i}\end{array} \text{ for $i=1,2,\cdots I+1$},
\end{align}
and the radiation condition in $u_1$ and $u_{I+1}$
\begin{align}
&u(\mathbf{r}) = \sum_{m,n} a_{mn}^u e^{i(\kappa_x^m x+\kappa_y^n y+k_u^{(m,n)}(z-z_u))}, z \geq z_u\\
&u(\mathbf{r}) = \sum_{m,n} a_{mn}^d e^{i(\kappa_x^m x+\kappa_y^n y-k_d^{(m,n)}(z-z_d))}, z \leq z_d, \label{eq:radiation}
\end{align}
where $\kappa_x^m = k_{1x}+2\pi m/d$, $\kappa_x^n = k_{1y}+2\pi n/d$, $k_u^{(m,n)} = \sqrt{k_1^2-(\kappa_x^m)^2-(\kappa_y^n)^2}$, and $k_d^{(m,n)} = \sqrt{k_{I+1}^2-(\kappa_x^m)^2-(\kappa_y^n)^2}$ and the sign of the square root is taken as positive real or positive imaginary. The existence and uniqueness of the solution of the BVP (\ref{eq:helmholtz})-(\ref{eq:radiation}) are briefly discussed in \cite{cho2015robust} including at Wood anomalies.

In Section \ref{sec:formulation}, we describe the periodized boundary integral representation and its discretization for the BVP (\ref{eq:helmholtz})-(\ref{eq:radiation}), where a new specialized quadrature is introduced in Section \ref{sec:zeta_trap}. Section \ref{sec:solver} describes the fast solution procedure for the discretized system by a reduction into block tridiagonal form. Several numerical examples are presented in Section \ref{sec:results} and the paper is concluded in Section \ref{sec:conclusion}.

\section{Boundary integral formulation, periodizing scheme, and its discretization}\label{sec:formulation}
From the periodizing idea in two-dimensions\cite{cho2015robust}, the scattered field or solution of the Helmholtz equation (\ref{eq:helmholtz}) in each layer is represented by the sum of near- and far-field contribution. The near-field contribution uses the free-space Green's function in an integral equation on the interfaces in the unit cell and its 8 immediate neighbors. The far-field contribution uses artificial (proxy) point sources on a sphere that is centered on and enclosing the unit cell.

We first define the standard single- and double-layer potentials \cite{colton1998inverse} for the Helmholtz equation residing on a general surface $\Gamma$ at wavenumber $k_i$ for the $i^\text{th}$ layer,
\begin{align}
(\cS_\Gamma^i\sigma)(\rr) :=\int_\Gamma G^i(\rr,\rr')\sigma(\rr')\,\d S_{\rr'},~~ & (\cD_\Gamma^i\tau)(\rr) :=\int_\Gamma \frac{\pd G^i}{\pd\nn'}(\rr,\rr')\tau(\rr')\,\d S_{\rr'},
\end{align}
where $\nn'$ is the unit normal on $\Gamma$ at $\rr'$, and where
\begin{equation}
G^i(\rr,\rr') := \frac{e^{ik_i|\rr-\rr'|}}{4\pi|\rr-\rr'|}
\label{eq:helmholtz_kernel}
\end{equation}
is the free-space Green's function at wavenumber $k_i$. The normal derivative of the potentials with respect to the unit normal vector $\nn$ at the target point $\rr$ are defined as
\begin{align}
(\cD_\Gamma^{i, *}\tau)(\rr) :=\int_\Gamma \frac{\pd G^i}{\pd\nn}(\rr,\rr')\tau(\rr')\,\d S_{\rr'},~~ & (\cT_\Gamma^{i}\tau)(\rr) :=\int_\Gamma \frac{\pd^2 G^i}{\pd\nn \pd\nn'}(\rr,\rr')\tau(\rr')\,\d S_{\rr'}.
\end{align}
Define the phased contribution from the nearest neighbors (indicated with a tilde)
\begin{align}
(\tilde{\cS}_\Gamma^i\sigma)(\rr) &:=\sum_{\substack{-1\leq l_x,l_y\leq1\\ \dd=d\,\langle l_x,l_y,0\rangle}}\alpha_x^{l_x}\alpha_y^{l_y}\int_\Gamma G^i(\rr,\rr'+\dd)\sigma(\rr')\,\d S_{\rr'} \label{psingle}\\
(\tilde{\cD}_\Gamma^i\tau)(\rr) &:=\sum_{\substack{-1\leq l_x,l_y\leq1\\ \dd=d\,\langle l_x,l_y,0\rangle}}\alpha_x^{l_x}\alpha_y^{l_y}\int_\Gamma \frac{\pd G^i}{\pd\nn'}(\rr,\rr'+\dd)\tau(\rr')\,\d S_{\rr'}.\label{pdouble}
\end{align}
The operators $\tilde{\cD}^{i,*}_\Gamma$ and $\tilde{\cT}^i_\Gamma$ are defined in the same manner. Furthermore, we will use two subscripts to denote target and source interfaces, for example,
\begin{equation}
(\tilde{\cS}_{\Gamma_t,\Gamma_s}^i\sigma)(\rr) :=\sum_{\substack{-1\leq l_x,l_y\leq1\\ \dd=d\,\langle l_x,l_y,0\rangle}}\alpha_x^{l_x}\alpha_y^{l_y}\int_{\Gamma_s} G^i(\rr,\rr'+\dd)\sigma(\rr')\,\d S_{\rr'},\quad \mathbf{r}\in\Gamma_t,
\end{equation}
represents the single layer operation defined in (\ref{psingle}) at the target interface $\Gamma_t$ due to the source interface $\Gamma_s$. Other source-target interaction operators, such as $\tilde{\cD}_{\Gamma_t,\Gamma_s}^i$, $\tilde{\cD}_{\Gamma_t,\Gamma_s}^{i, *}$, and $\tilde{\cT}_{\Gamma_t,\Gamma_s}^i$ are similarly defined.

Next we define the proxy points $\{\yy^i_p\}_{p=1}^P$ to be a spherical grid on the sphere $\mathbb{S}_i$ of radius $R$ and centered on the domain $\Omega_i$ (see Fig.~\ref{fig:periodization}(c)). Then the proxy basis functions for the $i^\text{th}$ layer are
\begin{equation}
\phi^i_p(\rr) := \frac{\pd G^i}{\pd\nn_p}(\rr,\yy^i_p) + ik_iG^i(\rr,\yy^i_p),\quad \rr\in\Omega_i,\quad p = 1,\dots,P
\end{equation}
where $\nn_p$ is the outward-pointing unit normal to $\mathbb{S}_i$ at $\yy^i_p$.

Combining the above definitions of near-field layer potentials and proxy basis, the ansatz for the scattered field in each layer is then given by
\begin{equation}
\begin{aligned}
u_1(\mathbf{r})&=\tilde{\cD}^1_{\Gamma_1}\tau_1+\tilde{\cS}^1_{\Gamma_1}\sigma_1+\sum_{p=1}^Pc^1_p\phi^1_p, & &\mathbf{r} \in \Omega_1,\\
u_i(\mathbf{r})&=\tilde{\cD}^i_{\Gamma_{i-1}}\tau_{i-1}+\tilde{\cS}^i_{\Gamma_{i-1}}\sigma_{i-1}+\tilde{\cD}^i_{\Gamma_i}\tau_i+\tilde{\cS}^i_{\Gamma_i}\sigma_i+\sum_{p=1}^Pc^i_p\phi^i_p, & &\mathbf{r} \in \Omega_i, i=2,3,\cdots I,\\
u_{I+1}(\mathbf{r})&=\tilde{\cD}^{I+1}_{\Gamma_I}\tau_I+\tilde{\cS}^{I+1}_{\Gamma_I}\sigma_I+\sum_{p=1}^Pc^{I+1}_p\phi^{I+1}_p, & &\mathbf{r} \in \Omega_{I+1}.
\end{aligned}\label{eq:ansatz}
\end{equation}
where the unknowns are the density functions $\tau_i$ and $\sigma_i$, $1\leq i \leq I$, and the proxy strengths $c^i:=\{c^i_p\}_{p=1}^P$, $1\leq i \leq I+1.$

The ansatz \eqref{eq:ansatz} must satisfy the continuity conditions at the interfaces $\Gamma_i$, the quasi-periodic conditions, and the radiation conditions. Following a similar procedure as in \cite{cho2019spectrally,cho2015robust}, one can substitute the ansatz \eqref{eq:ansatz} into (\ref{eq:continuity_cond}--\ref{eq:radiation}) and use the standard jump relations \cite[Thm 3.1]{colton1998inverse} to obtain a linear system
\begin{align}
\mathbf{A}\bs\eta+\mathbf{B}\mathbf{c} \phantom{\, + \mathbf{W}\mathbf{a} \,} &= \mathbf{f},\label{mat_eq:continuity}\\
\mathbf{C}\bs\eta+\mathbf{Q}\mathbf{c} \phantom{\, + \mathbf{W}\mathbf{a} \,} &= \mathbf{0},\label{mat_eq:periodic}\\
\mathbf{Z}\bs\eta+\mathbf{V}\mathbf{c} + \mathbf{W}\mathbf{a} &= \mathbf{0},\label{mat_eq:radiation}
\end{align}
where we will next define the matrices and vectors in the system.

The first two equations \eqref{mat_eq:continuity} and \eqref{mat_eq:periodic}, respectively, account for the continuity conditions at the interfaces and the quasi-periodic conditions, where $\mathbf{A}$ is an $I$-by-$I$ block-tridiagonal matrix with the nonzero blocks
\begin{equation}
\begin{aligned}
\mathbf{A}_{i,i} &= 
	\begin{bmatrix}
	\mathbf{I} + (\tilde{\cD}_{\Gamma_i,\Gamma_i}^i-\tilde{\cD}_{\Gamma_i,\Gamma_i}^{i+1}) & (\tilde{\cS}_{\Gamma_i,\Gamma_i}^i-\tilde{\cS}_{\Gamma_i,\Gamma_i}^{i+1})\\
	(\tilde{\cT}_{\Gamma_i,\Gamma_i}^i-\tilde{\cT}_{\Gamma_i,\Gamma_i}^{i+1}) & -\mathbf{I} + (\tilde{\cD}_{\Gamma_i,\Gamma_i}^{i,*}-\tilde{\cD}_{\Gamma_i,\Gamma_i}^{i+1,*})
	\end{bmatrix},\quad i=1,\dots,I,\\
\mathbf{A}_{i,i+1} &= 
	\begin{bmatrix}
	-\tilde{\cD}_{\Gamma_i,\Gamma_{i+1}}^{i+1} & -\tilde{\cS}_{\Gamma_i,\Gamma_{i+1}}^{i+1}\\
	-\tilde{\cT}_{\Gamma_i,\Gamma_{i+1}}^{i+1} & -\tilde{\cD}_{\Gamma_i,\Gamma_{i+1}}^{i+1,*}
	\end{bmatrix},\quad i=1,\dots,I-1,\\
\mathbf{A}_{i+1,i} &= 
	\begin{bmatrix}
	\tilde{\cD}_{\Gamma_{i+1},\Gamma_i}^{i+1} & \tilde{\cS}_{\Gamma_{i+1},\Gamma_i}^{i+1}\\
	\tilde{\cT}_{\Gamma_{i+1},\Gamma_i}^{i+1} & \tilde{\cD}_{\Gamma_{i+1},\Gamma_i}^{i+1,*}
	\end{bmatrix},\quad i=1,\dots,I-1,
\end{aligned}
\label{eq:A_block}
\end{equation}
where $\mathbf{B}$ is an $I$-by-$(I+1)$ block-bidiagonal matrix with the nonzero blocks
\begin{equation}
\begin{aligned}
\mathbf{B}_{i,i} &= \begin{bmatrix} \phi^{i}_1\big|_{\Gamma_{i}} & \dots & \phi^{i}_P\big|_{\Gamma_{i}} \\ \tfrac{\pd\phi^{i}_1}{\pd \nn}\big|_{\Gamma_{i}} & \dots & \tfrac{\pd\phi^{i}_P}{\pd \nn}\big|_{\Gamma_{i}} \end{bmatrix},
\,
\mathbf{B}_{i,i+1} = \begin{bmatrix} -\phi^{i+1}_1\big|_{\Gamma_{i}} & \dots & -\phi^{i+1}_P\big|_{\Gamma_{i}} \\ -\tfrac{\pd\phi^{i+1}_1}{\pd \nn}\big|_{\Gamma_{i}} & \dots & -\tfrac{\pd\phi^{i+1}_P}{\pd \nn}\big|_{\Gamma_{i}} \end{bmatrix},\, i=1,\dots,I,
\end{aligned}
\end{equation}
where $\mathbf{C}$ is an $(I+1)$-by-$I$ block-bidiagonal matrix with the nonzero blocks
\begin{equation}
\begin{aligned}
\mathbf{C}_{i,i} &= \sum_{\substack{-1\leq l \leq1 \\ \dd_x = d\langle 1,l,0\rangle\\ \dd_y = d\langle l,1,0\rangle}}\begin{bmatrix} 
\alpha^{-1}_x \cD^{i}_{R_i+\dd_x,\Gamma_i}-\alpha^2_x \cD^{i}_{L_i-\dd_x,\Gamma_i} & \alpha^{-1}_x \cS^{i}_{R_i+\dd_x,\Gamma_i}-\alpha^2_x \cS^{i}_{L_i-\dd_x,\Gamma_i} \\ 
\alpha^{-1}_x \cT^{i}_{R_i+\dd_x,\Gamma_i}-\alpha^2_x \cT^{i}_{L_i-\dd_x,\Gamma_i} & \alpha^{-1}_x \cD^{i,*}_{R_i+\dd_x,\Gamma_i}-\alpha^2_x \cD^{i,*}_{L_i-\dd_x,\Gamma_i} \\
\alpha^{-1}_y \cD^{i}_{B_i+\dd_y,\Gamma_i}-\alpha^2_y \cD^{i}_{F_i-\dd_y,\Gamma_i} & \alpha^{-1}_y \cS^{i}_{B_i+\dd_y,\Gamma_i}-\alpha^2_y \cS^{i}_{F_i-\dd_y,\Gamma_i} \\ 
\alpha^{-1}_y \cT^{i}_{B_i+\dd_y,\Gamma_i}-\alpha^2_y \cT^{i}_{F_i-\dd_y,\Gamma_i} & \alpha^{-1}_y \cD^{i,*}_{B_i+\dd_y,\Gamma_i}-\alpha^2_y \cD^{i,*}_{F_i-\dd_y,\Gamma_i}
\end{bmatrix},\, i=1,\dots,I,
\\
\mathbf{C}_{i+1,i} &= \sum_{\substack{-1\leq l \leq1 \\ \dd_x = d\langle 1,l,0\rangle\\ \dd_y = d\langle l,1,0\rangle}}\begin{bmatrix} 
\alpha^{-1}_x \cD^{i+1}_{R_{i+1}+\dd_x,\Gamma_{i}}-\alpha^2_x  \cD^{i+1}_{L_{i+1}-\dd_x,\Gamma_{i}} & \alpha^{-1}_x \cS^{i+1}_{R_{i+1}+\dd_x,\Gamma_{i}}-\alpha^2_x  \cS^{i+1}_{L_{i+1}-\dd_x,\Gamma_{i}}\\ 
\alpha^{-1}_x \cT^{i+1}_{R_{i+1}+\dd_x,\Gamma_{i}}-\alpha^2_x  \cT^{i+1}_{L_{i+1}-\dd_x,\Gamma_{i}} & \alpha^{-1}_x \cD^{i+1,*}_{R_{i+1}+\dd_x,\Gamma_{i}}-\alpha^2_x  \cD^{i+1,*}_{L_{i+1}-\dd_x,\Gamma_{i}} \\
\alpha^{-1}_y \cD^{i+1}_{B_{i+1}+\dd_y,\Gamma_{i}}-\alpha^2_y  \cD^{i+1}_{F_{i+1}-\dd_y,\Gamma_{i}} & \alpha^{-1}_y \cS^{i+1}_{B_{i+1}+\dd_y,\Gamma_{i}}-\alpha^2_y  \cS^{i+1}_{F_{i+1}-\dd_y,\Gamma_{i}} \\ 
\alpha^{-1}_y \cT^{i+1}_{B_{i+1}+\dd_y,\Gamma_{i}}-\alpha^2_y  \cT^{i+1}_{F_{i+1}-\dd_y,\Gamma_{i}} & \alpha^{-1}_y \cD^{i+1,*}_{B_{i+1}+\dd_y,\Gamma_{i}}-\alpha^2_y  \cD^{i+1,*}_{F_{i+1}-\dd_y,\Gamma_{i}}
\end{bmatrix},\, i=1,\dots,I,
\end{aligned}
\label{eq:C_block}
\end{equation}
and where $\mathbf{Q}$ is an $(I+1)$-by-$(I+1)$ block-diagonal matrix with the nonzero blocks
\begin{equation}
\mathbf{Q}_{i,i} = \begin{bmatrix} 
\phi^{i}_1\big|_{R_i} - \alpha_x \phi^{i}_1\big|_{L_i} & \dots & \phi^{i}_P\big|_{R_i} - \alpha_x \phi^{i}_P\big|_{L_i} \\
\tfrac{\pd\phi^{i}_1}{\pd n}\big|_{R_i} - \alpha_x \tfrac{\pd\phi^{i}_1}{\pd n}\big|_{L_i} & \dots & \tfrac{\pd\phi^{i}_P}{\pd n}\big|_{R_i} - \alpha_x \tfrac{\pd\phi^{i}_P}{\pd n}\big|_{L_i}\\
\phi^{i}_1\big|_{B_i} - \alpha_y \phi^{i}_1\big|_{F_i} & \dots & \phi^{i}_P\big|_{B_i} - \alpha_y \phi^{i}_P\big|_{F_i}\\
\tfrac{\pd\phi^{i}_1}{\pd n}\big|_{B_i} - \alpha_y \tfrac{\pd\phi^{i}_1}{\pd n}\big|_{F_i} & \dots & \tfrac{\pd\phi^{i}_P}{\pd n}\big|_{B_i} - \alpha_y \tfrac{\pd\phi^{i}_P}{\pd n}\big|_{F_i}
\end{bmatrix},\, i=1,\dots,I+1.
\label{eq:Q_block}
\end{equation}
The corresponding vectors in (\ref{mat_eq:continuity}-\ref{mat_eq:periodic}), including the densities $\bs\eta$ on the interfaces, the proxy source strengths $\mathbf{c}$ on the proxy spheres, and the right-hand side functions $\mathbf{f}$, are given by
\begin{equation}
\begin{aligned}
\bs\eta &= \begin{bmatrix} \bs\eta_1 & \dots & \bs\eta_I \end{bmatrix}^T,\, \mbox{ where }\bs\eta_i :=\begin{bmatrix} \tau_i \\ \sigma_i \end{bmatrix},\\
\mathbf{c} &= \begin{bmatrix} c^1 & \dots & c^{I+1} \end{bmatrix}^T,\, \mathbf{f} = \begin{bmatrix} -u^{inc}|_{\Gamma_1} & -\frac{\pd u^{inc}}{\pd\nn}|_{\Gamma_1} & 0 & \dots & 0 \end{bmatrix}^T.
\end{aligned}
\end{equation}

Equation \eqref{mat_eq:radiation} accounts for the radiation conditions at the artificial interfaces $U$ and $D$, where $\mathbf{Z}$ is a $2$-by-$I$ and $\mathbf{V}$ a $2$-by-$(I+1)$ block-sparse matrix, given by
\begin{equation}
\begin{aligned}
\mathbf{Z} &=
\begin{bmatrix}
\mathbf{Z}_U & \mathbf{0} & \dots & \mathbf{0}\\
\mathbf{0} & \dots & \mathbf{0} & \mathbf{Z}_D
\end{bmatrix},
\,
\mathbf{Z}_U = \begin{bmatrix}
\tilde{\cD}^1_{U,\Gamma_1} & \tilde{\cS}^1_{U,\Gamma_1} \\ \tilde{\cT}^1_{U,\Gamma_1} & \tilde{\cD}^{1,*}_{U,\Gamma_1}
\end{bmatrix},
\,
\mathbf{Z}_D = \begin{bmatrix}
\cD^{I+1}_{D,\Gamma_I} & \cS^{I+1}_{\cD,\Gamma_I} \\ \cT^{I+1}_{\cD,\Gamma_I} & \cD^{I+1,*}_{D,\Gamma_I}
\end{bmatrix},\\
\mathbf{V} &=
\begin{bmatrix}
\mathbf{V}_U & \mathbf{0} & \dots & \mathbf{0}\\
\mathbf{0} & \dots & \mathbf{0} & \mathbf{V}_D
\end{bmatrix},
\,
\mathbf{V}_U = \begin{bmatrix} \phi^1_1\big|_U & \dots & \phi^1_P\big|_U \\ \tfrac{\pd\phi^1_1}{\pd n}\big|_U &\dots & \tfrac{\pd\phi^1_P}{\pd n}\big|_U \end{bmatrix},
\,
\mathbf{V}_D = \begin{bmatrix} \phi^{I+1}_1\big|_D & \dots & \phi^{I+1}_P\big|_D \\ \tfrac{\pd\phi^{I+1}_1}{\pd n}\big|_D & \dots & \tfrac{\pd\phi^{I+1}_P}{\pd n}\big|_D \end{bmatrix},
\end{aligned}
\end{equation}
and where $\mathbf{W}$ is a $2$-by-$2$ block-diagonal matrix, in which the $(1,1)$-block $\mathbf{W}_U$ and the $(2,2)$-block $\mathbf{W}_D$ are given by
\begin{equation}
\mathbf{W}_U = \begin{bmatrix} -e^{i(\kappa_x^m x+\kappa_y^n y)}\big|_U \\ -ik_u^{(m,n)}e^{i(\kappa_x^m x+\kappa_y^n y)}\big|_U \end{bmatrix},
\,
\mathbf{W}_D = \begin{bmatrix} -e^{i(\kappa_x^m x+\kappa_y^n y)}\big|_D \\ ik_d^{(m,n)}e^{i(\kappa_x^m x+\kappa_y^n y)}\big|_D \end{bmatrix}, m,n = -K,-K+1,\dots,K-1,K.
\end{equation}
The corresponding vector $\mathbf{a}=[\mathbf{a}^u\,\,\,\mathbf{a}^d]^T$ contains the Rayleigh-Block coefficients such that $\mathbf{a}^u = [a_{mn}^u]$ and $\mathbf{a}^d = [a_{mn}^d]$, $m,n = -K,-K+1,\dots,K-1,K.$

\subsection{Discretization of functions and operators}\label{sec:discretize}

To accurately solve the system (\ref{mat_eq:continuity}-\ref{mat_eq:radiation}), we describe high-order collocation methods for the discretization of the integral operators in $\mathbf{A},\mathbf{C},\mathbf{Z}$ and the functions in the rest of the matrix blocks.

We choose collocation points on the interfaces and the walls as follows. On each side wall $W$, where $W$ is one of $R_i,L_i,B_i$, and $F_i$, for $i=1,\dots,I+1$, we sample $M_w$ points $\{\mathbf{x}_m^W\}_{m=1}^{M_w}$ that correspond to a 2-D tensor product of 1-D quadrature nodes. For the top and bottom walls $U$ and $D$, we use $M$ equally-spaced nodes (associated with the double Trapezoidal rule) $\{\mathbf{x}_m^U\}_{m=1}^M$ on $U$ and $\{\mathbf{x}_m^D\}_{m=1}^M$ on $D$. For the interfaces $\Gamma_i$ which are smooth and doubly periodic, we assume the parameterizations $g_i$ on the rectangle $[0,1]^2$ are given by 
\begin{equation}
\Gamma_i = \left\{\mathbf{r}=(x,y,z)\,|\,z=g_i(x,y), (x,y)\in[0,1]^2\right\}.
\label{eq:parameterization}
\end{equation}
Fixing a set of $N$ equally-spaced nodes $\{(x_n,y_n)\}_{n=1}^N\subset[0,1]^2$ associated with the double Trapezoidal rule, we can then sample $\Gamma_i$ at $N$ points $\{\mathbf{x}_n^i=(x_n,y_n,g_i(x_n,y_n))\}_{n=1}^N$; let $\{w_n^i\}_{n=1}^N$ be the associated quadrature weights such that 
\begin{equation}
\int_{\Gamma_i}f(\mathbf{r})\,\d S_\mathbf{r} \approx \sum_{n=1}^Nf(\mathbf{x}_n^i)w_n^i
\label{eq:smooth_quadr}
\end{equation}
holds to high accuracy for any given smooth periodic function $f$ on $\Gamma_i$.

The discretization of the matrix blocks in the system (\ref{mat_eq:continuity}-\ref{mat_eq:radiation}) becomes straightforward with the collocation points and the quadrature above (except for the diagonal blocks of $\mathbf{A}$, which we will address shortly). For example, in the $\mathbf{Q}_{i,i}$ block in \eqref{eq:Q_block}, an entry $\phi^{i}_p\big|_{R_i} - \alpha_x \phi^{i}_p\big|_{L_i}$ in the first row is replaced by an $M_w\times P$ matrix
$$
\Big(\phi^{i}_p(\mathbf{x}_m^{R_i}) - \alpha_x \phi^{i}_p(\mathbf{x}_m^{L_i})\Big),\quad m=1,\dots,M_w,\, p=1,\dots,P.
$$
All the entries in $\mathbf{B}, \mathbf{Q}, \mathbf{V}$, and $\mathbf{W}$ can be discretized similarly. On the other hand, the operators in $\mathbf{A}, \mathbf{C}, \mathbf{Z}$ (except for the diagonal blocks of $\mathbf{A}$) can be discretized using the smooth quadrature \eqref{eq:smooth_quadr}. For example in the $\mathbf{C}_{i,i}$ block in \eqref{eq:C_block}, an entry $\alpha^{-1}_x S^{i}_{R_i+\dd_x,\Gamma_i}-\alpha^2_x S^{i}_{L_i-\dd_x,\Gamma_i}$ can be replaced by an $M_w\times N$ matrix
$$
\Big(\alpha^{-1}_x G^i(\mathbf{x}_m^{R_i}+\dd_x,\mathbf{x}_n^i)w_n^i-\alpha^2_x G^i(\mathbf{x}_m^{L_i}-\dd_x,\mathbf{x}_n^i)w_n^i\Big),\quad m=1,\dots,M_w,\, n=1,\dots,N.
$$

We now consider the discretization of the diagonal blocks $\mathbf{A}_{i,i}$ in \eqref{eq:A_block}. These blocks involve interactions from $\Gamma_i$ to itself, where the involved integrals become singular, so the smooth quadrature ceases to be accurate. We will focus on discretizing the entries involving the single-layer operator $\tilde{S}$; other self-interactions involving the operators $\tilde{D}, \tilde{D}^*$, and $\tilde{T}$ can then be discretized similarly. Consider $\tilde{S}_{\Gamma_1,\Gamma_1}^1-\tilde{S}_{\Gamma_1,\Gamma_1}^{2}$ in the $\mathbf{A}_{1,1}$ block, the associated integral operator is
\begin{equation}
\Big((\tilde{\cS}_{\Gamma_1,\Gamma_1}^1-\tilde{\cS}_{\Gamma_1,\Gamma_1}^{2})\sigma\Big)(\rr) =\sum_{\substack{-1\leq l_x,l_y\leq1\\ \dd=d\langle l_x,l_y,0\rangle}}\alpha_x^{l_x}\alpha_y^{l_y}\int_{\Gamma_1} \Big(G^1(\rr,\rr'+\dd)-G^{2}(\rr,\rr'+\dd)\Big)\sigma(\rr')\,\d S_{\rr'},\, \mathbf{r}\in\Gamma_1,
\label{eq:SLP_diag_block}
\end{equation}
which consists of contributions from $\Gamma_1+\dd$ to $\Gamma_1$. When $\mathbf{d}\neq\mathbf{0}$, one can still discretize \eqref{eq:SLP_diag_block} using the smooth quadrature \eqref{eq:smooth_quadr} with nodes $\{\mathbf{x}_n^1+\mathbf{d}\}_{n=1}^N$ and weights $\{w_n^1\}_{n=1}^N$. When $\mathbf{d}=\mathbf{0}$, applying the smooth quadrature will result in a matrix
\begin{equation}
\Big(S_{\Gamma_1,\Gamma_1}^1-S_{\Gamma_1,\Gamma_1}^{2}\Big)_{m,n} = \Big(G^1(\mathbf{x}_m^1,\mathbf{x}_n^1)-G^{2}(\mathbf{x}_m^1,\mathbf{x}_n^1)\Big)w_n^1,\quad m,n=1,\dots,N
\label{eq:SLP_naive_trapezoid}
\end{equation}
whose diagonal entries are infinite; zeroing out the diagonal entries (i.e., using the ``punctured Trapezoidal rule'') will allow the discretization to converge as $N\to\infty$, but only very slowly. We next describe a new quadrature that makes corrections near the diagonal of \eqref{eq:SLP_naive_trapezoid} to obtain a high-order discretization.

\subsection{Special quadrature for self interaction}\label{sec:zeta_trap}

We describe a specialized quadrature for \eqref{eq:SLP_diag_block} which modifies the entries of \eqref{eq:SLP_naive_trapezoid}. This quadrature is based on the error-corrected Trapezoidal quadrature method \cite{wu2022unified}, which is a recent generalization of \cite{wu2021corrected} that achieves high-order accuracies. For a fixed target point $\rr_0:=\mathbf{x}_m^1\in\Gamma_1$ (where $m$ is fixed), consider the integral
\begin{equation}
I = \int_{\Gamma_1} \Big(G^1(\rr_0,\rr)-G^{2}(\rr_0,\rr)\Big)\sigma(\rr)\,\d S_{\rr}
\label{eq:SLP1_integral}
\end{equation}
and its boundary-corrected, punctured Trapezoidal rule approximation
\begin{equation}
Q_h = \sum_{\substack{n=1\\n\neq m}}^N \Big(G^1(\mathbf{x}_m^1,\mathbf{x}_n^1)-G^{2}(\mathbf{x}_m^1,\mathbf{x}_n^1)\Big)w_n^1\sigma(\mathbf{x}_n^1) + C_h.
\end{equation}
where $h$ is the grid spacing, and where the boundary correction $C_h$ is a linear combination of the values and derivatives of the integrand in \eqref{eq:SLP1_integral}; the coefficients of the linear combination appears in the two-dimensional Euler-Maclaurin formula (see, for example, \cite[Theorem 2.6]{lyness1976error}) and do not depend on the integrand. The exact form of $C_h$ is unimportant for our purpose, because the surface $\Gamma_1$ is periodic thus the boundary errors will vanish. To analyze the error of $I - Q_h$, we only need to assume that $C_h$ is a sufficiently high-order boundary correction so that the dominant error always comes from the singularity of the integrand.

\begin{lemma}
The error of the Trapezoidal rule approximation of \eqref{eq:SLP1_integral}, $E_h := I - Q_h$, has an asymptotic expansion
\begin{equation}
E_h \sim c_1h^3 + c_2h^5 + \dots = \sum_{p=1}^\infty c_ph^{2p+1},\quad \text{as }h\to0,
\label{eq:SLP_error}
\end{equation}
where the coefficients $c_1,c_2,\dots$ only depend on the following information at the target point $\rr_0$: the derivatives of the parameterization of the surface $\Gamma_1$ at $\rr_0$, and the value and derivatives of a smooth function $\varphi(\rr)$ at $\rr_0$, where $\varphi(\rr)$ can be explicitly constructed from the integrand of \eqref{eq:SLP1_integral}.
\end{lemma}
\begin{proof}
Let $r:=|\rr_0-\rr|$, then using the definition \eqref{eq:helmholtz_kernel} and the parameterization, the integral \eqref{eq:SLP1_integral} can be rewritten as
\begin{equation}
\begin{aligned}
I &= \int_0^1\int_0^1 \Big(r\,\psi_c(r) + i\psi_s(r)\Big)\sigma(\rr)\,J(\rr)\,\d x\,\d y\\
 &= \int_0^1\int_0^1 r\,\psi_c(r)\,\sigma(\rr)\,J(\rr)\,\d x\,\d y + i\int_0^1\int_0^1\psi_s(r)\sigma(\rr)\,J(\rr)\,\d x\,\d y\\
 &:= I_1 + iI_2
\end{aligned}
\end{equation}
where $\rr=\rr(x,y)$ is the parameterization \eqref{eq:parameterization} for $\Gamma_1$, where $J(\rr)$ is the Jacobian, and where
\begin{equation*}
\psi_c(r):=\frac{\cos(k_1r)-\cos(k_2r)}{r^2}\quad \text{and}\quad \psi_s(r):=\frac{\sin(k_1r)-\sin(k_2r)}{r}
\end{equation*}
are smooth functions of $r^2=|\rr_0-\rr|^2$, thus are also smooth functions of $\rr$ on $\Gamma_1$. The Trapezoidal rule approximation $Q_h=Q_{h,1}+iQ_{h,2}$ where $Q_{h,1}$ and $Q_{h,2}$ are the approximation of the integrals $I_1$ and $I_2$, respectively. The imaginary component $I_2$ is smooth thus the approximation $Q_{h,2}$ is high-order accurate, so the error $E_h \sim I_1-Q_{h,1}$. The integrand of $I_1$ can be written as $r\,\varphi(\rr)$, where
\begin{equation}
\varphi(\rr):= \psi_c(r)\sigma(\rr)J(\rr)
\end{equation}
is a smooth function of $\rr(x,y)$. Following \cite[Section 3.1]{wu2022unified}, the integrand of $I_1$ has the following expansion at $\rr_0=\rr(x_0,y_0)$
\begin{equation}
r\,\varphi(\rr) \sim \sum_{m=0}^\infty \sum_{n=3m}^\infty\sum_{\substack{l_1,l_2\geq0\\ l_1+l_2=n}}\alpha^{m,n}_{l_1,l_2}\frac{\hat{x}^{l_1}\hat{y}^{l_2}}{\cF(\hat{x},\hat{y})^{m-1/2}}
\end{equation}
where $\hat{x}=x-x_0$ and $\hat{y}=y-y_0$, where the coefficients $\alpha^{m,n}_{l_1,l_2}$ depend on the value and derivatives of $\phi$ and $\rr(x,y)$ at $(x_0,y_0)$, and where $\cF(x,y)$ is the first fundamental form of $\Gamma_1$ at $\rr_0$ defined as
\begin{equation}
\cF(\hat{x},\hat{y})=E\hat{x}^2+2F\hat{x}\hat{y}+G\hat{y}^2,\, E=|\rr_x(x_0,y_0)|^2, F=\rr_x(x_0,y_0)\cdot\rr_y(x_0,y_0), G=|\rr_y(x_0,y_0)|^2.
\end{equation}
Then applying the generalized Euler-Maclaurin formula \cite[Theorem 3.2]{wu2022unified}, we have
\begin{equation}
E_h \sim I_1-Q_{h,1} \sim \sum_{m=0}^\infty \sum_{\substack{n=3m\\ n\text{ even}}}^\infty\sum_{\substack{l_1,l_2\geq0\\ l_1+l_2=n}}-\alpha^{m,n}_{l_1,l_2}C^m_{l_1,l_2}[E,F,G]h^{n-2m+3}
\label{eq:SLP_error_detailed}
\end{equation}
where $C^m_{l_1,l_2}[E,F,G]$ are coefficients that only depend on $E,F,G$ and can be computed based on \cite[Theorem 3.3]{wu2022unified}. To finish the proof, the equation \eqref{eq:SLP_error_detailed} can be rearranged into the form \eqref{eq:SLP_error} by grouping the terms by the powers of $h$.
\end{proof}

\begin{figure}[htbp]
\includegraphics[width=\textwidth]{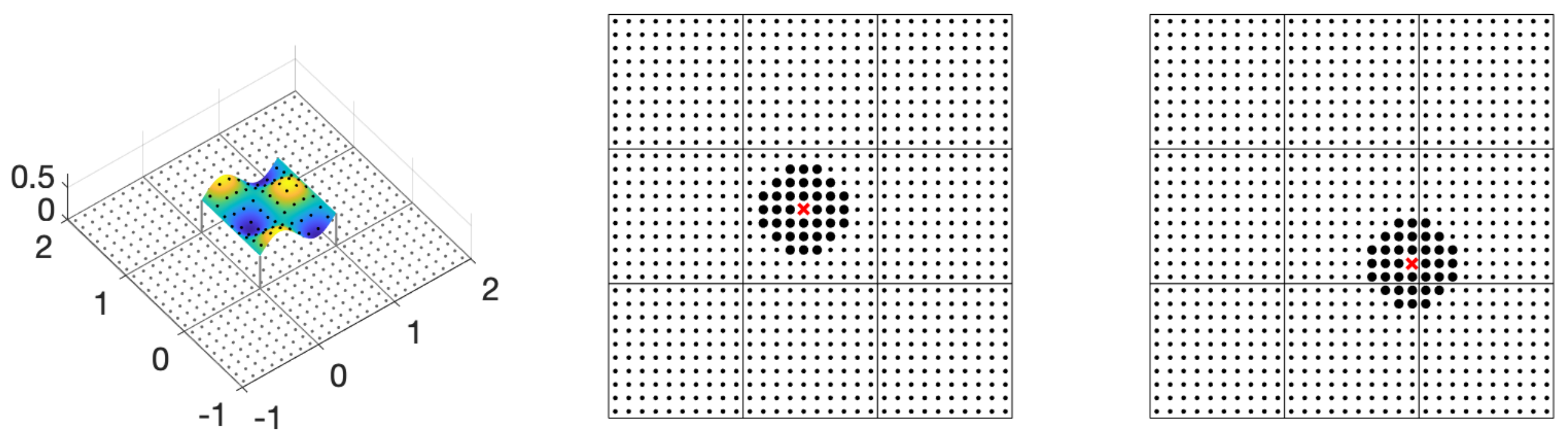}
\caption{Left: A doubly periodic surface with period $d=1$ and parameterized over $[0,1]^2$, sampled with equally-spaced nodes $\{\mathbf{x}_n^i\}_{n=1}^N$. The nodes in the parameter space, $\{(x_n,y_n)\}_{n=1}^N\subset [0,1]^2$ in the $xy$-plane, together with their copies in the 8 nearest neighbors, are shown in all three panels. Middle: local error-correction stencil (thick dots) around a target point (red cross) in the middle of the domain. Right: error-correction stencil around a target point near the edge of the domain; note that the stencil extends outside the domain to the near periodic copies.}\label{fig:interface_grid}
\end{figure}

A high-order corrected Trapezoidal rule for the integral operator \eqref{eq:SLP1_integral} can then be constructed by adding a sufficient number of terms in the error expansion \eqref{eq:SLP_error}. However, a direct computation of \eqref{eq:SLP_error} requires approximating the higher derivatives of $\varphi$ and $\rr(x,y)$. To avoid computing these higher derivatives, the \emph{moment-fitting procedure} from \cite[Section 3.2]{wu2022unified} is used to fit the error \eqref{eq:SLP_error} on a local stencil around the target point $\rr_0$, this procedure only requires evaluating $\varphi(\rr), \rr(x,y)$ and the first fundamental form $\cF(\hat{x},\hat{y})$ on the stencil. Fig.~\ref{fig:interface_grid} shows the stencil for a $5^\text{th}$ order correction; note that when the stencil extends outside the central domain, the corresponding weights should be constructed with appropriate phased contribution involving the Bloch phases $\alpha_x$ and $\alpha_y$. We refer to \cite{wu2022unified} for more details on higher order discretization.

\section{Rearrangement of equations, block elimination, and fast solver}\label{sec:solver}
The system (\ref{mat_eq:continuity}-\ref{mat_eq:radiation}) is sparse and can be transformed into a block-tridiagonal system as follows.

First, rearrange the rows and columns of (\ref{mat_eq:continuity}-\ref{mat_eq:radiation}) to form the block $2\times2$ system
\begin{equation}
\begin{bmatrix}
\mathbf{A} & \tilde{\mathbf{B}}\\
\tilde{\mathbf{C}} & \tilde{\mathbf{Q}}
\end{bmatrix}
\begin{bmatrix}
\bs{\eta} \\
\tilde{\mathbf{c}}
\end{bmatrix}
=
\begin{bmatrix}
\mathbf{f} \\
\mathbf{0}
\end{bmatrix},
\label{eq:2x2system}
\end{equation}
where one combines $\mathbf{c}_1$ and $\mathbf{a}^u$, and respectively $\mathbf{c}_{I+1}$ and $\mathbf{a}^d$, to form 
\begin{equation}
\tilde{\mathbf{c}}_1 = \begin{bmatrix} \mathbf{c}_1 \\ \mathbf{a}^u \end{bmatrix},\quad
\tilde{\mathbf{c}}_{I+1} = \begin{bmatrix} \mathbf{c}_{I+1} \\ \mathbf{a}^d \end{bmatrix},\quad
\tilde{\mathbf{c}} = \begin{bmatrix} \tilde{\mathbf{c}}_1 & \mathbf{c}_2 & \dots  & \mathbf{c}_{I} & \tilde{\mathbf{c}}_{I+1}\end{bmatrix}^T
\end{equation}
and, accordingly, $\tilde{\mathbf{Q}}$ is formed by replacing the $(1,1)$-block and the $(I+1,I+1)$-block of $\mathbf{Q}$ with
\begin{equation}
\tilde{\mathbf{Q}}_{1,1} = \begin{bmatrix} \mathbf{Q}_{1,1} & \mathbf{0}\\ \mathbf{V}_U & \mathbf{W}_U \end{bmatrix},\quad \tilde{\mathbf{Q}}_{I+1,I+1} = \begin{bmatrix} \mathbf{Q}_{I+1,I+1} & \mathbf{0}\\ \mathbf{V}_D & \mathbf{W}_D \end{bmatrix}.
\end{equation}
Likewise, $\tilde{\mathbf{B}}$ and $\tilde{\mathbf{C}}$ are formed, respectively, by replacing the $(1,1)$-block and the $(I,I+1)$-block of $\mathbf{B}$, and replacing the $(1,1)$-block and the $(I+1,I)$-block of $\mathbf{C}$, with
\begin{equation}
\begin{aligned}
\tilde{\mathbf{B}}_{1,1} &= \begin{bmatrix} \mathbf{B}_{1,1} & \mathbf{0} \end{bmatrix},\quad \tilde{\mathbf{B}}_{I,I+1} = \begin{bmatrix} \mathbf{B}_{I,I+1} & \mathbf{0} \end{bmatrix},\\
\tilde{\mathbf{C}}_{1,1} &= \begin{bmatrix} \mathbf{C}_{1,1} \\ \mathbf{Z}_U \end{bmatrix},\quad \tilde{\mathbf{C}}_{I+1,I} = \begin{bmatrix} \mathbf{C}_{I+1,I} \\ \mathbf{Z}_D \end{bmatrix}.
\end{aligned}
\end{equation}

Next, eliminate the unknowns $\tilde{\mathbf{c}}$ to reduce the system \eqref{eq:2x2system} into block tridiagonal form
\begin{equation}
\tilde{\mathbf{A}}\bs\eta = \mathbf{f},
\label{eq:tridiag_system}
\end{equation}
where $\tilde{\mathbf{A}}$ is an $I\times I$ block tridiagonal matrix with the nonzero blocks 
\begin{align}
\tilde{\mathbf{A}}_{1,1} &=\mathbf{A}_{1,1}-\tilde{\mathbf{B}}_{1,1}\tilde{\mathbf{Q}}_1^\dagger \tilde{\mathbf{C}}_{1,1}-\mathbf{B}_{1,2}\mathbf{Q}_{2,2}^\dagger\mathbf{C}_{2,1},\\
\tilde{\mathbf{A}}_{i,i} &=\mathbf{A}_{i,i}-\mathbf{B}_{i,i}\mathbf{Q}_i^\dagger\mathbf{C}_{i,i}-\mathbf{B}_{i,i+1}\mathbf{Q}_{i+1,i+1}^\dagger\mathbf{C}_{i+1,i}, & &i=2,3,\dots,I-1,\\
\tilde{\mathbf{A}}_{I,I} &=\mathbf{A}_{I,I}-\tilde{\mathbf{B}}_{I,I}\tilde{\mathbf{Q}}_I^\dagger \tilde{\mathbf{C}}_{I,I}-\mathbf{B}_{I,I+1}\mathbf{Q}_{I+1,I+1}^\dagger\mathbf{C}_{I+1,1},\\
\tilde{\mathbf{A}}_{i,i+1} &=\mathbf{A}_{i,i+1}-\mathbf{B}_{i,i+1}\mathbf{Q}_{i+1,i+1}^\dagger\mathbf{C}_{i+1,i+1}, & &i=1,2,3,\dots,I-1,\\
\tilde{\mathbf{A}}_{i+1,i} &=\mathbf{A}_{i+1,i}-\mathbf{B}_{i+1,i+1}\mathbf{Q}_{i+1,i+1}^\dagger\mathbf{C}_{i+1,i}, & &i=1,2,3,\dots,I-1.
\end{align}
where $^\dagger$ denotes the pseudo inverse of a rectangular matrix. The system \eqref{eq:tridiag_system} can be solved efficiently using the block LU factorization \cite[Sec.4.5.1]{GoluVanl96} in $O(N^3I)$ operations, assuming the number of unknowns associated with each interface is $O(N)$. The algorithm proceeds by first initializing $\mathbf{f}'_1=\mathbf{f}_1$ and $\tilde{\mathbf{A}}'_{1,1}=\tilde{\mathbf{A}}_{1,1}$, then a \emph{forward sweep} for $i=2$ to $I$,
$$
\begin{aligned}
\tilde{\mathbf{A}}'_{i,i}&= \tilde{\mathbf{A}}_{i,i}-\tilde{\mathbf{A}}_{i,i-1}(\tilde{\mathbf{A}}'_{i-1,i-1})^{-1}\tilde{\mathbf{A}}_{i-1,i},\\
\mathbf{f}'_{i}&= \mathbf{f}_{i}-\tilde{\mathbf{A}}_{i,i-1}(\tilde{\mathbf{A}}'_{i-1,i-1})^{-1}\mathbf{f}'_{i-1},
\end{aligned}
$$
followed by solving $\tilde{\mathbf{A}}'_{I,I}\bs\eta_I=\mathbf{f}'_I$ for $\bs\eta_I$, and then a \emph{backward sweep} for $i=I-1$ down to $1$ to solve for each $\bs\eta_i$,
$$
\bs\eta_i = (\tilde{\mathbf{A}}'_{i,i})^{-1}(\mathbf{f}'_i-\tilde{\mathbf{A}}_{i,i+1}\bs\eta_{i+1}).
$$

Once the densities $\bs\eta_i$ are obtained, the modified proxy strengths $\tilde{\mathbf{c}}_i$ can be recovered by
\begin{align}
\tilde{\mathbf{c}}_1 &=-\tilde{\mathbf{Q}}_{1,1}^\dagger\tilde{\mathbf{C}}_{1,1}\bs\eta_1,\\
\mathbf{c}_i &=-\tilde{\mathbf{Q}}_{i,i}^\dagger(\tilde{\mathbf{C}}_{i,i-1}\bs\eta_{i-1}+\tilde{\mathbf{C}}_{i,i}\bs\eta_{i}), & i=2,3,\dots,I,\\
\tilde{\mathbf{c}}_{I+1} &=-\tilde{\mathbf{Q}}_{I+1,I+1}^\dagger\tilde{\mathbf{C}}_{I+1,I}\bs\eta_{I}.
\end{align}
Then substituting the densities $\bs\eta_i$ and proxy strengths $\mathbf{c}_i$ into the ansatz \eqref{eq:ansatz} gives the scattered field at any locations in any given layer $\Omega_i$.

\section{Numerical Results}\label{sec:results}

In this section, we present numerical examples of multilayered media scattering using the numerical scheme described in the previous sections. We first investigate the convergence of the solutions and the complexity of the computational time and memory requirements, then we show examples of scattering with many layers. Most computations are performed on a Mac Pro with 3.2 GHz 16-Core Intel Xeon W processors and 192 GB RAM using MATLAB R2022a; the 101-layer example in Fig.~\ref{fig:many_layer_101} is performed on a workstation with 3.1 GHz 36-Core Intel Xeon Gold 6254 processors and 768 GB of RAM. Points on all the surrounding walls are uniformly distributed. The proxy points for each layer are placed on a sphere of radius $1.5$ enclosing the unit cell of the layer. (See Fig.~\ref{fig:periodization}(b)-(c).) 

We use the relative flux error $E_{\text{flux}}$ as an independent measure of accuracy, which is defined as
\begin{equation}
E_{\text{flux}}:= \left|\frac{\sum_{m,n} k_u^{(m,n)}|a_{mn}^u|^2+\sum_{m,n} k_d^{(m,n)}|a_{mn}^d|^2-k_1\cos\phi^{inc}}{k_1\cos\phi^{inc}}\right|.
\end{equation}
This measure is based on the conservation of flux (energy) \cite{cho2019spectrally}.

\paragraph{Convergence of the solutions} The first example considers the transmission problem with a two-layered media, where the interface $\Gamma_1$, parameterized by $g_1(x,y)=h\sin(2\pi x)\cos(2\pi y)$, is either flat ($h=0$) or corrugated ($h=0.2$). The wavenumbers in the top and bottom layers are $k_1=8$ and $k_2=16$ and the incident angle is $\phi^{inc}=\frac{5\pi}{6}$ and $\theta^{inc}=0$. In Fig.~\ref{fig:conv_2layer_flat}, we test the convergence of the scattered field $u$ at the point $(-0.25,-0.25,0.25)$ in the top layer (i.e. the reflected field) and at the point $(-0.25,-0.25,-0.25)$ in the bottom layer (i.e. the transmitted field), and investigate the relative flux error $E_{\text{flux}}$. The convergence against $N$, the number of points per interface, and against $P$, the number of proxy points per layer, are shown, where the analytic solution is used as the reference solution in the flat case and the numerical solution with $N=120^2$ and $P=3120$ is used as the reference solution in the corrugated case. Other parameters, if not specified, are fixed at $N=120^2$, $M_w=M=25^2$, $P=3120$. Quadrature corrections of $7^\text{th}$ order are used.
\begin{figure}[t] 
\includegraphics[width=\textwidth]{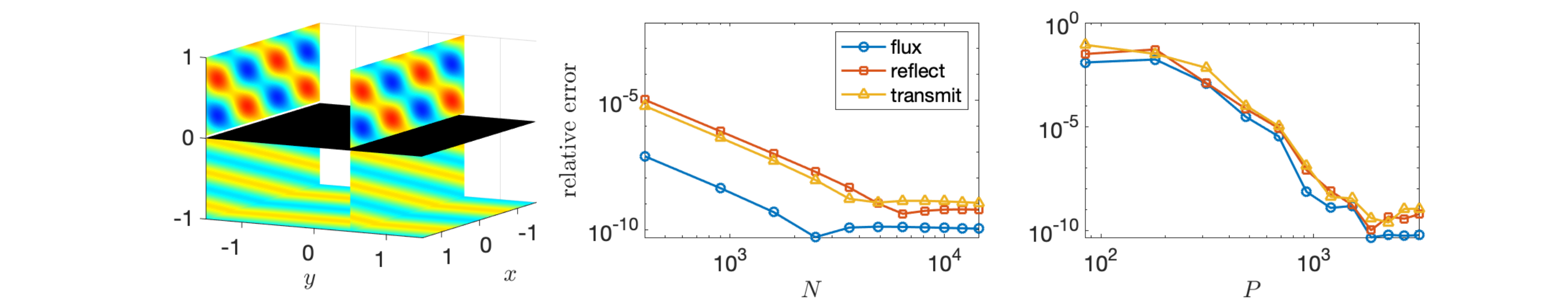}\\
\includegraphics[width=\textwidth]{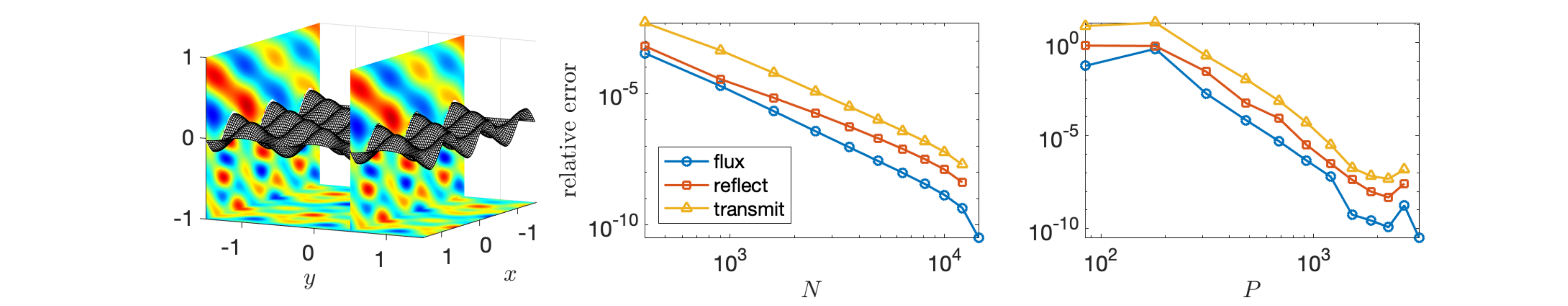}
\caption{Two-layered media transmission problem with wavenumbers $k_1=8$ and $k_2=16.$ The incident angle is $\phi^{inc}=\frac{5\pi}{6}$ with $\theta^{inc}=0$. Results with a flat interface are on the top row, and with a corrugated interface on the bottom row.  Left panels: Total field. Middle panels: relative error of the reflected wave (at the point $(-0.25,-0.25,0.25)$) and of the transmitted wave (at the point $(-0.25,-0.25,-0.25)$) and the flux error against $N$, the number of points per interface. Right panels: convergence and flux error against $P$, the number of proxy points per layer.}\label{fig:conv_2layer_flat}
\end{figure}

In Fig.~\ref{fig:conv_rough_ord}, we perform additional convergence study similar to Fig.~\ref{fig:conv_2layer_flat} by changing the level of roughness of the interface and the order of quadrature correction. As expected, we observe that a higher order method ($7^\text{th}$ order) converges faster and provides more digits of accuracy than a lower order method ($5^\text{th}$ order) at the same number of discretization points. When the interface is a highly corrugated ($h=0.5$), more points (larger $N$) are required than a less corrugated interface ($h=0.2$) to reach the same level of accuracy.
\begin{figure}[htbp] 
\includegraphics[width=\textwidth]{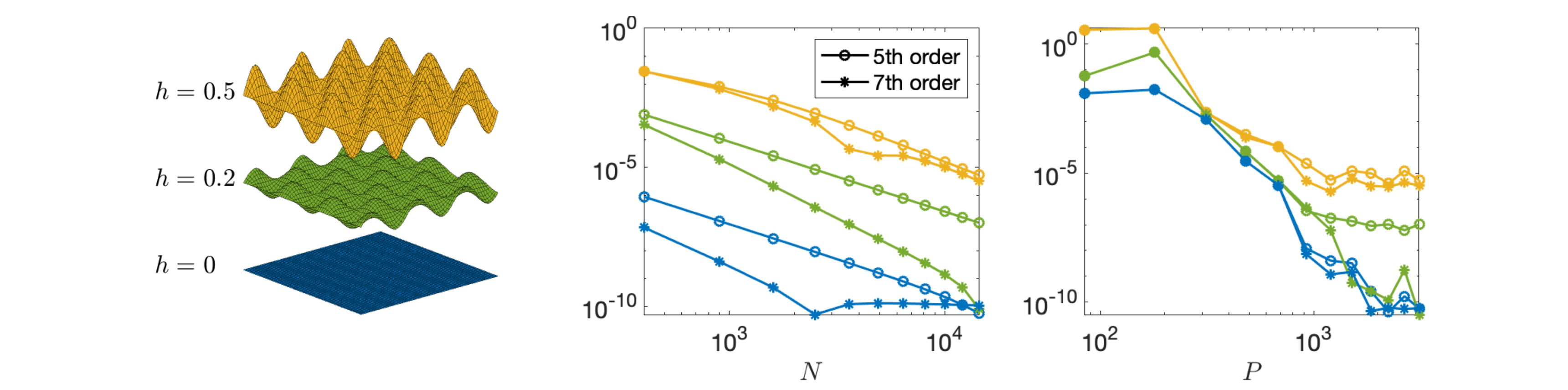}
\caption{Investigation of the effects of different levels of roughness of the interface and the order of the quadrature correction on the convergence of the two-layered media problem in Fig.~\ref{fig:conv_2layer_flat}. Left panel illustrates the surface $g_1(x,y)=h\sin(2\pi x)\cos(2\pi y)$ for three different values of $h$. (Note: this is \emph{not} the actual arrangement of intefaces.) Convergence of flux error against $N$ (middle panel) and against $P$ (right panel) are shown for the cases $h=0$ (blue), $h=0.2$ (green) and $h=0.5$ (yellow), with quadrature corrections of $5^\text{th}$ order (circles) and of $7^\text{th}$ order (asterisks). All other parameters are the same as in Fig.~\ref{fig:conv_2layer_flat}}\label{fig:conv_rough_ord}
\end{figure}

\paragraph{Computational time and memory requirement} With the sparse system \eqref{eq:tridiag_system}, the required computational time and memory are expected to scale \emph{linearly} against the number of interfaces $I$. In Table \ref{table:scaling}, we show the computational results for media with $1$ to $30$ interfaces, where $T_\text{pre}$ is the time for precomputing the quadrature correction weights, $T_\text{fill}$ the time and $M_\text{fill}$ the memory required for filling the matrices in the system (\ref{mat_eq:continuity}-\ref{mat_eq:radiation}), and $T_\text{solve}$ the total time for both the reduction to and the solution of the block tridiagonal system \eqref{eq:tridiag_system} via block LU factorization. In all cases, the interfaces $\Gamma_i$ are parameterized by $g_i(x,y)=0.2\sin(2\pi x)\cos(2\pi y)-(i-1)$ and discretized using $N=60^2$ points each and with $7^\text{th}$-order quadrature. The wavenumbers $k_i$ alternate between $10$ and $20$. Other parameters are fixed at $\phi^{inc}=5\pi/6$, $\theta^{inc}=\pi/4$, $P=2380$, $M_w=M=20^2$ and $K=10$. The relative flux error stays below $5\times10^{-6}$ for any number of layers presented.
\begin{table}[htbp]
\centering
\begin{tabular}{l|lllll}
$I$ & $T_\text{pre}$ (s) & $T_\text{fill}$ (s) & $T_\text{solve}$ (s) & $M_\text{fill}$  (GB) & $E_\text{flux}$ \\ 
\hline
1 & 4 & 23 & 16 & 2.0 & 8.5e-08 \\ 
2 & 8 & 65 & 43 & 5.2 & 3.1e-06 \\ 
3 & 13 & 107 & 73 & 8.4 & 2.9e-06 \\ 
4 & 17 & 149 & 103 & 11.7 & 4.8e-06 \\ 
5 & 21 & 191 & 133 & 14.9 & 4.6e-06 \\ 
10 & 42 & 390 & 280 & 31.0 & 4.9e-06 \\ 
20 & 84 & 797 & 575 & 63.3 & 4.9e-06 \\ 
30 & 125 & 1203 & 875 & 95.6 & 4.9e-06
\end{tabular}
\caption{Time (in seconds) and memory (in gigabytes)  requirements for multilayered media scattering, where $I$ is the number of interfaces.}\label{table:scaling}
\end{table}

\paragraph{Example with 101 layers} To demonstrate the capability of the algorithm, Fig.~\ref{fig:many_layer_101} shows the total field across a 101-layered media (cross-section in the $xz$-plane). The wavenumbers $k_i$ are randomly chosen between $8$ and $20$ for each layer. A relative flux error of $7.4\times10^{-6}$ is achieved with the parameters $N=60^2$, $P=2380$, $M_w=M=20^2$, $K=10$, and with $7^\text{th}$ order quadrature corrections. The total number of unknowns is about 961.3k. 
The computation is completed in about $1.83$ hours (0.2 hours for precomputing the quadrature correction weights, 1.28 hours for filling the matrices and 0.35 hours for solving the system) and used about 321 GB of memory.
\begin{figure}[htbp] 
\centering
\includegraphics[width=\textwidth]{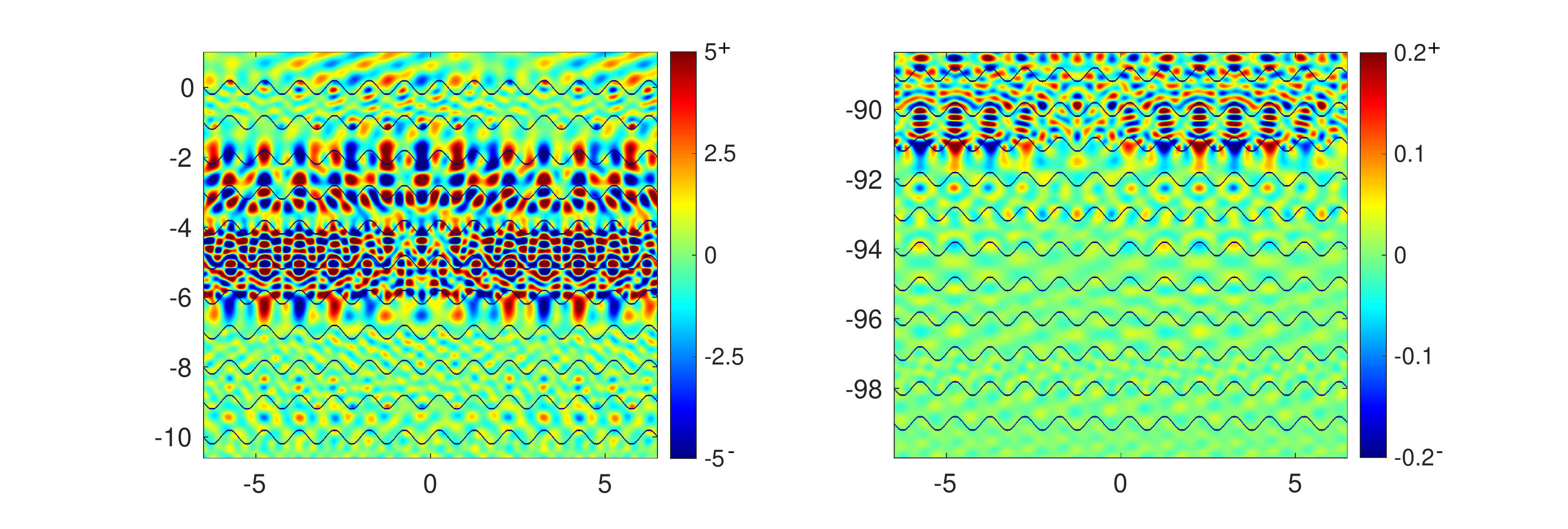}
\caption{101-layered media scattering. The left panel shows the real part of the total field in the top 12 layers and the right panel in the bottom 12 layers (note the different color scales). The wavenumbers for each layer are chosen randomly in $[8,20]$. $N=60^2$ points per interface and $P=2380$ proxy points per layer are used. Computation is done in $1.83$ hours and requires 321 GB of RAM. Relative flux error $\sim7.4\times10^{-5}$.
}\label{fig:many_layer_101}
\end{figure}

\paragraph{Transmission and reflection spectra} Finally, we compute the transmission and reflection spectra, for a range of incident angles $\pi/2<\phi^{inc}<3\pi/2$ and fixed $\theta^{inc}=0$, for an 11-layered media with flat or corrugated interfaces. $\Gamma_i$ are parameterized by $g_i(x,y)=h\sin(2\pi x)\cos(2\pi y)-(i-1)$ for $1\leq i\leq10$. In Fig.~\ref{fig:spectra}, we show the spectra for $h=0$ (flat interfaces) and $h=0.2$ (corrugated interfaces). The wavenumbers $k_i$ alternate between $2\pi$ and $4\pi$ (i.e., alternating between $1$ and $2$ wavelengths across layers). We used $N=40^2$ points for each interface and $P=1740$ proxy points for each layer, so that the flux error is below $10^{-4}$ in all cases for all incident angles. We observe that when the interfaces are corrugated, the spectra vary faster with the incident angle than when the interfaces are flat. The computations are accelerated by precomputing and storing the matrix components that are independent of the incident angle or the Bloch phases, at the cost of higher memory requirements. This gives an over 4x speedup and the total computational time for each structure takes about 1 hour.
\begin{figure}[htbp] 
\includegraphics[width=\textwidth]{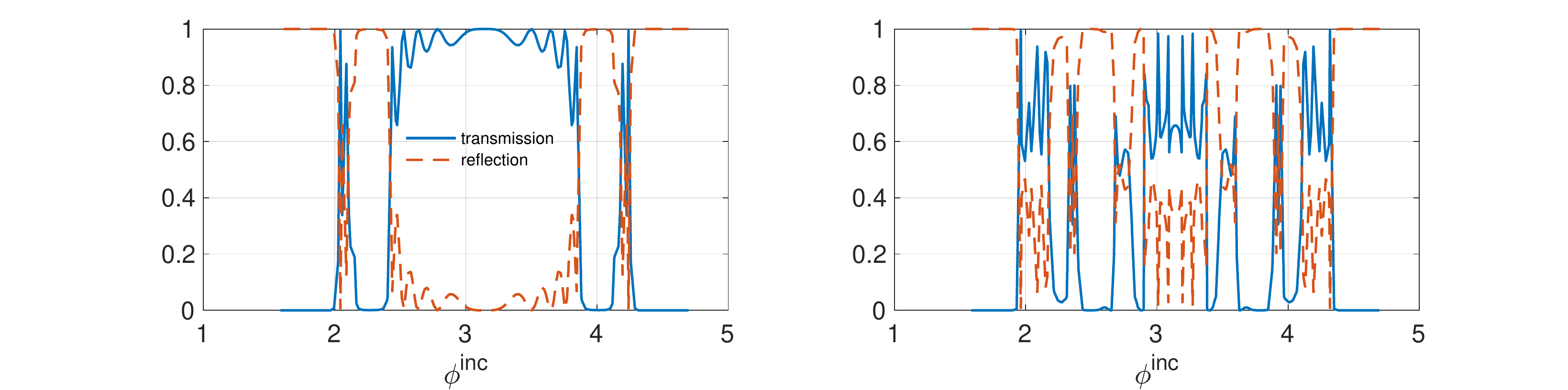}
\caption{The transmission and reflection spectra of an 11-layered media, where the wavenumbers $k_i$ of each layer alternates between $2\pi$ and $4\pi$. Left panel: spectra for $h=0$ (flat interfaces). Right panel: spectra for $h=0.2$ (corrugated interfaces).}\label{fig:spectra}
\end{figure}

\section{Conclusion}\label{sec:conclusion}
A new 3-D multilayered media solver for doubly-periodic geometry is presented. The solver is fast, accurate, and robust, equipped with new high-order specialized quadrature. An accuracy of $5$ to $10$ digits can be obtained for structures with multiple layers with a small number of unknowns for each interface. The solver is capable of handling a structure with 101 layers in under 2 hours. The computational time and the memory requirement scale linearly with the number of layers in the structure. The transmission and reflection spectra for media with many layers can be computed efficiently, which is useful for applications in science and engineering. Future directions include further accelerating the solver by developing fast direct solvers in the style of \cite{zhang2022fast} and developing new solvers for Maxwell's equations.

\section*{Acknowledgments}
B. Wu thanks Per-Gunnar Martinsson for generously allowing use of his workstation for the 101-layered media example. M.H. Cho and B. Wu thank Alex Barnett for his help to identify the collaboration for this work. The work of M.H. Cho is supported by NSF grant DMS 2012382

\bibliography{sample}
\bibliographystyle{abbrv}

\end{document}